%% file: Ext_computability.tex
\documentclass[12pt]{amsart}

\usepackage[pdfauthor   = {Mohamed\ Barakat\ and\ Markus\ Lange-Hegermann},
            pdftitle    = {},
            pdfsubject  = {},
            pdfkeywords = {Serre\ quotients, reflective\ localizations\ of\ Abelian\ categories, Yoneda\ Ext, Ext-computability;},
            bookmarks=true,
            bookmarksopen=true,
            pagebackref=true,
            hyperindex=true,
            colorlinks=true,
            linkcolor=blue,
            citecolor=blue,
            filecolor=blue,
            urlcolor=blue,
            ]{hyperref}

\include{header}

\author{Mohamed Barakat}
\address{Department of mathematics, University of Kaiserslautern, 67653 Kaiserslautern, Germany}
\email{\href{mailto:Mohamed Barakat <barakat@mathematik.uni-kl.de>}{barakat@mathematik.uni-kl.de}}

\author{Markus Lange-Hegermann}
\address{Lehrstuhl B f\"ur Mathematik, RWTH Aachen University, 52062 Aachen, Germany}
\email{\href{mailto:Markus Lange-Hegermann <markus.lange.hegermann@rwth-aachen.de>}{markus.lange.hegermann@rwth-aachen.de}}

\begin{document}

\title[On the \texorpdfstring{$\Ext$}{Ext}-computability of \nameft{Serre} quotient categories]{On the \texorpdfstring{$\Ext$}{Ext}-computability of \nameft{Serre} quotient categories}

\begin{abstract}
  To develop a constructive description of $\Ext$ in categories of coherent sheaves over certain schemes, we establish a binatural isomorphism between the $\Ext$-groups in \nameft{Serre} quotient categories $\A/\C$ and a direct limit of $\Ext$-groups in the ambient \nameft{Abel}ian category $\A$.
  For $\Ext^1$ the isomorphism follows if the thick subcategory $\C \subset \A$ is localizing.
  For the higher extension groups we need further assumptions on $\C$.
  With these categories in mind we cannot assume $\A/\C$ to have enough projectives or injectives and therefore use \nameft{Yoneda}'s description of $\Ext$.
\end{abstract}


\maketitle

\renewcommand\theenumi{\alph{enumi}}
\renewcommand{\labelenumi}{(\theenumi)}

\input{Ext_computability_content.tex}

\input{Ext_computability.bbl}

\end{document}

%% file: header.tex
\usepackage[utf8,utf8x]{inputenc}
\usepackage[english]{babel}
\usepackage[T1]{fontenc}
\usepackage{geometry}                
\usepackage{times}
\usepackage[novbox]{pdfsync}
\usepackage{mathrsfs}
\usepackage{latexsym}
\usepackage{amssymb}
\usepackage{amsthm}
\usepackage{epsfig}
\usepackage{colortbl}
\usepackage[all]{xy}
\usepackage{fancyvrb}
\usepackage{graphicx}
\usepackage[dvipsnames]{xcolor}
\usepackage{accents} 
\usepackage{enumerate}

\renewcommand\theenumi{\alph{enumi}}
\renewcommand{\labelenumi}{(\theenumi)}

\usepackage{wrapfig}
\usepackage{tikz}
\usetikzlibrary{shapes,arrows,matrix,backgrounds,positioning,plotmarks,calc,patterns,matrix,
decorations.shapes,
decorations.fractals,
decorations.markings,
decorations.pathreplacing,
decorations.pathmorphing,
decorations.text
}
\usepackage[colorinlistoftodos,shadow]{todonotes}

\usepackage{multirow}
\usepackage{mdwlist}
\usepackage{stmaryrd}
\usepackage{mathdots} 

\usepackage[toc,page]{appendix}
\usepackage{float}

\newtheoremstyle{mytheoremstyle} 
    {5pt}                    
    {5pt}                    
    {\itshape}                   
    {\parindent}                           
    {\bf}                   
    {.}                          
    {.5em}                       
    {}  

\theoremstyle{mytheoremstyle}

\newtheorem{theorem}{Theorem}[section]

\newtheorem{lemm}[theorem]{Lemma}
\newtheorem{prop}[theorem]{Proposition}

\newtheoremstyle{mytdefintionstyle} 
    {5pt}                    
    {5pt}                    
    {\rm}                   
    {\parindent}                           
    {\bf}                   
    {.}                          
    {.5em}                       
    {}  

\theoremstyle{remark}
\newtheorem{rmrk}[theorem]{Remark}

\theoremstyle{mytdefintionstyle}
\newtheorem{defn}[theorem]{Definition}
\newtheorem{exmp}[theorem]{Example}

\newtheoremstyle{exmp_contd}
    {5pt}                    
    {5pt}                    
    {\rm}                   
    {\parindent}                           
    {\bf}                   
    {.}                          
    {.5em}                       
    {\thmname{#1}\ \thmnumber{ #2}\thmnote{#3}\ (continued)}  
\theoremstyle{exmp_contd}

\input{pre.tex}

\definecolor{darkgreen}{rgb}{0.008,0.617,0.067}
\definecolor{brown}{rgb}{0.6,0.4,0.2}

\newif\ifjournalversion

%% file: pre.tex
\newcommand\nameft\textrm

\newcommand{\QQ}{{\mathscr{Q}}}
\renewcommand{\SS}{{\mathscr{S}}}

\DeclareMathOperator{\Ext}{Ext}

\DeclareMathOperator{\coker}{coker}

\DeclareMathOperator{\img}{im}

\DeclareMathOperator{\Hom}{Hom}

\DeclareMathOperator{\Spec}{Spec}

\DeclareMathOperator{\Sat}{Sat}

\newcommand{\Coh}{\mathfrak{Coh}\,}

\renewcommand\O{\mathcal{O}}
\newcommand\PP{\mathbb{P}}

\newcommand\A{\mathcal{A}}
\newcommand\C{\mathcal{C}}

\newcommand\B{\mathcal{B}}

\renewcommand{\O}{\mathcal{O}}

\newcommand{\Z}{\mathbb{Z}}

\renewcommand\phi{\varphi}

\DeclareMathOperator\Id{Id}

\DeclareMathOperator\cores{co-res}

\definecolor{darkgray}{rgb}{0.3,0.3,0.3}

\newcommand{\FillArea}[4]{#1 -- #2 -- #3 -- #4 -- #1}

\pgfarrowsdeclarecombine[\pgflinewidth]
  {doublestealth}{doublestealth}{stealth'}{stealth'}{stealth'}{stealth'}

%% file: Ext_computability_content.tex
\section{Introduction}

Our original motivation is to develop a constructive and computer-friendly description of \nameft{Abel}ian categories of coherent sheaves $\Coh X$ on various classes of \nameft{Noether}ian schemes $X$.
In this setup the functors $\Hom$ and $\Ext^c$ are ubiquitous, and any constructive approach needs to incorporate these functors.
For example, the global section functor on $\Coh X$ can be defined as $\Gamma:=\Hom(\O_X,-)$, i.e., in terms of the $\Hom$ functor and the structure sheaf $\O_X$.
The higher sheaf cohomology $H^i$ is usually defined in the nonconstructive larger category of quasi-coherent sheaves on $X$ as $H^i=R^i\Gamma=\Ext^i(\O_X,-)$.

In this paper we deal with computing the bivariate $\Ext^i(-,-)$, where for the special univariate case of sheaf cohomology $H^i=\Ext^i(\O_X,-)$ there often exist good algorithms.
Our minimal assumption on $X$ is that the category $\Coh X$ is equivalent to a \nameft{Serre} quotient category $\A/\C \simeq \Coh X$ where $\A$ is a computable category (in the sense of Appendix~\ref{sec:computability}) of finitely presented graded modules and $\C\subset\A$ is its thick subcategory of all modules with zero sheafification.
The canonical functor $\QQ:\A\to\A/\C$ then plays the role of the exact sheafification functor $\operatorname{Sh}:\A \to \Coh X, M \mapsto \widetilde{M}$.
Some classes of schemes for which this holds are listed in \cite[Section~4]{BL_SerreQuotients}, including projective and toric schemes.

The computability of $\Ext^c$ would usually follow from that of $\Hom$ in case the underlying category is computable and has constructively enough projectives or enough injectives.
However, as categories of \emph{coherent} sheaves do not in general admit enough injectives or projectives we cannot assume this for the computation of $\Ext^c$ in an abstract \nameft{Serre} quotient $\A/\C$.
Hence, $\Ext^c$ in such an $\A/\C$ cannot even be defined constructively as a derived functor using projective or injective resolutions and we are left over with \nameft{Yoneda}'s description of $\Ext^c$ \cite{Oort_Yoneda64}.
Although \nameft{Yoneda}'s description does not a priori provide an algorithm to compute $\Ext^c$, it is sufficient to prove our main result: Under certain assumptions on $\C$ the computability of $\Ext^c$ in $\A/\C$ can be reduced to the computability of $\Ext^c$ in $\A$, provided a certain (infinite) direct limit is constructive.
More precisely:

\begin{theorem} \label{thm:main}
  If $\C$ is an almost split localizing\footnote{Cf.~Definition~\ref{defn:almost_split}.} subcategory of an \nameft{Abel}ian category $\A$ then the binatural transformation\footnote{We drop the canonical functor $\QQ$ in $\Ext^c_{\A/\C}(\QQ(M),\QQ(N))$ since $\QQ$ is the identity on objects.}
  \[
    \QQ^{\Ext}: \varinjlim_{\substack{ M' \leq M,\\ M/M' \in \C}} \Ext^c_\A(M',N) \to \Ext^c_{\A/\C}(M,N)
  \]
  is an isomorphism (of \nameft{Abel}ian groups) for all $\C$-torsion-free $M \in \A$ and $\C$-saturated $N \in \A$.
\end{theorem}

For applications to coherent sheaves $\A/\C \simeq \Coh X$ (for $X$ as above) we need to prove that the thick subcategory $\C$ of modules with zero sheafification is almost split localizing.
This is for example the case if $X$ is a projective space and $\A$ is suitably chosen (see Section~\ref{sec:Hom} and Example~\ref{exmp:almost_split}).
However, it is worth mentioning that Theorem~\ref{thm:main} cannot cover\footnote{contrary to Theorem~\ref{thm:ext1}.} the case of coherent sheaves on nonsmooth toric varieties.
One can see this easily since the \nameft{Cox} ring and hence the category $\A$ of finitely presented graded modules over this ring is of finite global dimension while one can easily construct coherent sheaves on a nonsmooth toric variety with non-vanishing $\Ext^c$ for arbitrarily high $c$ (see Example~\ref{exmp:non_almost_split}).

The theorem suggests an algorithmic approach to the computability of $\Ext^c$ in $\A/\C$.
To compute the left hand side $\varinjlim_{\substack{ M' \leq M,\\ M/M' \in \C}} \Ext^c_\A(M',N)$ we need to be able to compute $\Ext^c_\A$ and a direct limit of \nameft{Abel}ian groups.
For categories of graded modules $\A$ there are well-known algorithms to compute $\Ext^c_\A$.
Proving that the (infinite) direct limit can be computed in finite terms depends on $\A$ and $\C$.
For example, in the category $\A/\C \simeq \Coh X$ of coherent sheaves on a projective space the finiteness of this direct limit follows from the \nameft{Castelnuovo-Mumford} regularity.
Thus, Theorem~\ref{thm:main} is an abstract form of \cite[Theorem 1]{SmithExt} and \cite{SmithExtOW}, without the context-specific convergence analysis.
If $\A$ is the category of graded modules and if the limit is reached for a certain $M' \leq M$ then one can use a graded free resolution of $M'$ in $\A$ to compute $\Ext^c_{\A/\C}(M,N)$.
In this case this graded free resolution of $M'$ in $\A$ corresponds, under the canonical functor $\QQ:\A \to \A/\C$, to a locally free resolution of $M$ in $\A/\C = \Coh X$ satisfying some regularity bounds.
We believe that the multigraded \nameft{Castelnuovo-Mumford} \cite{MS04,MS05,HSS} can be used to prove the finiteness of the limit in the case of smooth projective toric varieties.
We leave this for future work.

We briefly recall the language of \nameft{Serre} quotient categories in Section~\ref{sec:Serre} and deal with the $c=0$ case of Theorem~\ref{thm:main} in Section~\ref{sec:Hom}.
In Section~\ref{sec:Yoneda} we recall \nameft{Yoneda}'s description of $\Ext^c$ and in Section~\ref{sec:QExt} we define the binatural transformation $\QQ^{\Ext}$.
In the main Section~\ref{sec:proof} we define the above mentioned condition which $\C$ needs to satisfy and prove Theorem~\ref{thm:main}.
There it is also proved that the theorem is valid if $c=1$ under the weaker condition that $\C$ is a localizing subcategory of $\A$ (cf.~Theorem~\ref{thm:ext1}).
Finally, in Appendix~\ref{sec:computability} we briefly sketch a constructive context for this paper.

\section{Preliminaries on \nameft{Serre} quotients}\label{section:preliminaries_sere_quotients} \label{sec:Serre}

In this section we recall some results about \nameft{Serre} quotients \cite{Gab_thesis}.
From now on $\A$ is an \nameft{Abel}ian category.

A non-empty full subcategory $\mathcal{C}$ of an \nameft{Abel}ian category $\A$ is called \textbf{thick} if it is closed under passing to subobjects, factor objects, and extensions.
In this case the \textbf{(\nameft{Serre}) quotient category} $\A/\C$ is a category with the same objects as $\A$ and $\Hom$-groups
\[
  \Hom_{\A/\C}(M,N) := \varinjlim_{\substack{M' \le M, N' \le N,\\ M/M', N' \in \C}} \Hom_\A(M',N/N')\mbox{.}
\]
The \textbf{canonical functor} $\QQ:\A \to \A/\C$ is defined to be the identity on objects and maps a morphism $\phi \in \Hom_\A(M,N)$ to its class in the direct limit $\Hom_{\A/\C}(M,N)$.
The category $\A/\C$ is \nameft{Abel}ian and the canonical functor $\QQ: \A \to \A/\C$ is exact.

Let $\C\subset\A$ be thick.
An object $M \in \A$ is called \textbf{$\C$-torsion-free} if $M$ has no nonzero subobjects in $\C$.
We will need the following simple lemma.
\begin{lemm} \label{lemm:extension_of_C_free_is_C_free}
  An extension of $\C$-torsion-free objects is again $\C$-torsion-free.
\end{lemm}
\begin{proof}
  Let $E$ be an object in $\A$ with $\C$-torsion-free subobject $L$ and $\C$-torsion-free factor object $B=E/L$.
  Assume that $E$ has a nontrivial $\C$-subobject $T$.
  Since $L \cap T = 0$ we conclude that $T$ is isomorphic to the nontrivial $\C$-subobject $(T+L) / L$ of $B$, a contradiction.
\end{proof}
If every object $M \in \A$ has a \textbf{maximal $\C$-subobject} $H_\C(M)$ then we call $\C$ a \textbf{thick torsion} subcategory.
An object $M \in \A$ is called \textbf{$\C$-saturated} if it is $\C$-torsion-free and every extension of $M$ by an object $C \in \C$ is trivial.
Denote by $\Sat_\C(\A)\subset\A$ the full subcategory of $\C$-saturated objects with embedding functor $\iota: \Sat_\C(\A) \hookrightarrow \A$.
The thick subcategory $\C\subset\A$ is called a \textbf{localizing} subcategory if the canonical functor $\QQ: \A \to \A/\C$ admits a right adjoint $\SS:\A/\C \to \A$, called the \textbf{section functor} of $\QQ$.
The section functor $\SS:\A/\C \to \A$ is left exact and preserves products, the counit of the adjunction $\delta: \QQ \circ \SS \xrightarrow{\sim} \Id_{\A/\C}$ is a natural isomorphism.
Let $\eta: \Id_{\A} \to \SS \circ \QQ$ denote the unit of the adjunction.
The kernel $\ker\left( \eta_M:M \to (\SS \circ \QQ)(M) \right)$ is then the maximal $\C$-subobject $H_\C(M)$ of $M$.
The cokernel of $\eta_M$ lies in $\C$.
We call $(\SS \circ \QQ)(M)$ the \textbf{$\C$-saturation of $M$}.
An object $M$ in $\A$ is \textbf{$\C$-saturated} if and only if $\eta_M$ is an isomorphism.
The image $\SS(\A/\C)$ of $\SS$ is a subcategory of $\Sat_\C(\A)$ and the inclusion functor $\SS(\A/\C)\hookrightarrow \Sat_\C(\A)$ is an equivalence of categories with the restricted-corestricted monad $\SS\circ \QQ:\Sat_\C(\A) \to \SS(\A/\C)$ as a quasi-inverse.
The restricted canonical functor $\QQ:\Sat_\C(\A) \to \A/\C$ and the corestricted section functor $\SS:\A/\C \to \Sat_\C(\A)$ are quasi-inverse equivalences of categories.
In particular, $\Sat_\C(\A) \simeq \SS(\A/\C) \simeq \A/\C$ is an \nameft{Abel}ian category.
Define the reflector\footnote{A functor is called a reflector if it has a fully faithful right adjoint.} $\widehat{\QQ} := \cores_{\Sat_\C(\A)} (\SS \circ \QQ): \A \to \Sat_\C(\A)$.
The adjunction $\widehat{\QQ} \dashv (\iota: \Sat_\C(\A) \hookrightarrow \A)$ corresponds under the above equivalence to the adjunction $\QQ \dashv (\SS: \A/\C \to \A)$.
They both share the same adjunction monad $\SS \circ \QQ = \iota \circ \widehat{\QQ}: \A \to \A$.
In particular, the reflector $\widehat{\QQ}$ is exact and $\iota$ is left exact.
$\SS(\A/\C) \simeq \Sat_\C(\A)$ are not in general \nameft{Abel}ian \emph{sub}categories of $\A$, as short exact sequences in $\Sat_\C(\A)$ are not necessarily exact in $\A$.
For more details and for a characterization of the monad $\SS \circ \QQ$ see \cite{BL_Monads}.

\section{The \texorpdfstring{$c=0$}{c=0} case} \label{sec:Hom}

If the thick subcategory $\C \subset \A$ is torsion then the double direct limit in the definition of the $\Hom$-groups in $\A/\C$ simplifies to a single direct limit
\[
  \Hom_{\A/\C}(M,N) = \varinjlim_{\substack{M' \le M \\ M / M' \in \C}} \Hom_\A(M',N/H_\C(N)) \mbox{.}
\]
If furthermore $\C \subset \A$ is localizing then the $\Hom$-adjunction\footnote{Again, we drop the canonical functor $\QQ$ in $\Hom_{\A/\C}(\QQ(M),\QQ(N))$ since $\QQ$ is the identity on objects.} between $\QQ$ and $\SS$ yields
\begin{equation} \label{eq:Hom} \tag{Hom}
  \Hom_{\A}(M,(\SS\circ\QQ)(N)) \cong \Hom_{\A/\C}(M,N) \mbox{,}
\end{equation}
for all $M,N\in\A$, avoiding the direct limit completely.
Theorem~\ref{thm:main} generalizes this last formula, being the $c=0$ case.

The monad $\SS \circ \QQ$ together with its unit are constructive in the case $\A/\C \simeq \Coh \PP^n_k$, i.e., of coherent sheaves  on the projective space $X=\PP^n_k$ over a field $k$.
Hence, the above mentioned $\Hom$-adjunction can be used to compute (global) $\Hom$-groups.
More precisely, let $\A$ be the category of finitely presented $\Z$-graded $k[x_0,\ldots,x_n]$-modules generated in degree $\ge 0$ and $\C$ be the thick subcategory of finite length modules.
The $\C$-saturation of an $N \in \A$ is the truncated module of twisted global sections, i.e., $(\SS\circ\QQ)(N) = \bigoplus_{i\ge0} \Gamma(\widetilde{N}(i))$, where $\widetilde{N}\in\Coh \PP^n_k$ is the sheafifications of $N$.
For $X=\PP^n_k$, and hence for any projective scheme, there are already several algorithms to compute the monad $\SS \circ \QQ$; e.g., as an ideal transform \cite[Theorem~20.3.15]{BrSh}, or by the \nameft{Beilinson} monad \cite{Bei78,EFS,DE}, or by the BGG-correspondence.

Recently, \nameft{Perling} \cite{PerLift} described the section functor $\SS$ and hence the monad $\SS \circ \QQ$ for a larger class of schemes, but in a (yet) nonconstructive way.
A constructive description is highly desirable as it would widen the applicability of Theorem~\ref{thm:main} as an algorithm to compute $\Ext$'s for further classes of smooth schemes.

\section{\nameft{Yoneda}'s description of \texorpdfstring{$\Ext$}{Ext}} \label{sec:Yoneda}

Since in applications to coherent sheaves the quotient category $\A/\C$ does not have enough projectives or injectives we use \nameft{Yoneda}'s description of $\Ext^c$ (cf.~\cite[Section~III.5]{ML95}).

So let $\B$ be an \nameft{Abel}ian category.
A $c$-cocycle in the $\Ext^c_\B(M,N)$ group ($c>0$) is an equivalence class of $c$-extensions of $M$ by $N$, i.e., exact $\B$-sequences
  \[
    e:\  0 \xleftarrow{} M \xleftarrow{} G_c \xleftarrow{} \cdots \xleftarrow{} G_1 \xleftarrow{} N \xleftarrow{} 0
  \]
of length $c+2$.
Two $c$-extensions $e,e'$ of $M$ by $N$ are in directed relation if there exists a chain morphism of the form
\begin{center}
\begin{tikzpicture}
  \matrix (m) [matrix of math nodes,row sep=1em,column sep=0.9em]
  { 
    0 & M & G_c & \cdots & G_1 & N & 0 \\
    0 & M & G_c' & \cdots & G_1' & N & 0 \\
  };
  \path[-stealth']
  (m-1-7) edge (m-1-6)
  (m-1-6) edge (m-1-5)
  (m-1-5) edge (m-1-4)
  (m-1-4) edge (m-1-3)
  (m-1-3) edge (m-1-2)
  (m-1-2) edge (m-1-1)
  (m-2-7) edge (m-2-6)
  (m-2-6) edge (m-2-5)
  (m-2-5) edge (m-2-4)
  (m-2-4) edge (m-2-3)
  (m-2-3) edge (m-2-2)
  (m-2-2) edge (m-2-1)
  (m-1-2) edge [double distance=2pt,-] (m-2-2)
  (m-1-3) edge (m-2-3)
  (m-1-5) edge (m-2-5)
  (m-1-6) edge [double distance=2pt,-] (m-2-6)
  ;
  \node at ($(m-1-1)+(-0.6,0)$) {$e:$};
  \node at ($(m-2-1)+(-0.6,0)$) {$e':$};
\end{tikzpicture}
\end{center}
For $c>1$ this directed relation is not symmetric.
A $c$-cocycle is an equivalence class of the equivalence relation generated by this directed relation.
Abusing the notation we will denote by $e$ the $c$-cocycle in $\Ext_\B^c(M,N)$ of a $c$-extension $e$ of $M$ by $N$.

We now recall the definition of the \nameft{Yoneda} composition $\Ext_\B^c(M,N) \times \Ext_\B^{c'}(N,L) \to \Ext_\B^{c+c'}(M,L)$.
We start with the case $c,c' >0$.
For $e^M_N \in \Ext_\B^c(M,N)$ and $e^N_L \in \Ext_\B^{c'}(N,L)$ represented by the extensions
\[
  e^M_N:\  0 \xleftarrow{} M \xleftarrow{} G_c \xleftarrow{} \cdots \xleftarrow{} G_1 \xleftarrow{} N \xleftarrow{} 0
  \mbox{ and }
  e^N_L:\  0 \xleftarrow{} N \xleftarrow{} G_{c'}' \xleftarrow{} \cdots \xleftarrow{} G_1' \xleftarrow{} L \xleftarrow{} 0
  \mbox{,}
\]
respectively.
The \nameft{Yoneda} composite $e^M_L = e^M_N e^N_L \in \Ext_\B^{c+c'}(M,L)$ is the $(c+c')$-cocycle represented by the $(c+c')$-extension
\[
    e^M_L:\  0 \xleftarrow{} M \xleftarrow{} G_c \xleftarrow{} \cdots \xleftarrow{} G_1 \xleftarrow{} G_{c'}' \xleftarrow{} \cdots \xleftarrow{} G_1' \xleftarrow{} L \xleftarrow{} 0
\]  
of $M$ by $L$, where the morphism $G_1 \xleftarrow{} G_{c'}'$ is the composition $G_1 \xleftarrow{} N \xleftarrow{} G_{c'}'$.

For $c=0$ and $c'>0$ let $\phi^M_N \in \Hom_\B(M,N)$ and $e^N_L \in \Ext_\B^{c'}(N,L)$ as above.
The \nameft{Yoneda} composite $\phi^M_N e^N_L \in \Ext_\B^{c'}(M,L)$ is given by the pullback $c'$-extension
\begin{center}
\begin{tikzpicture}
  \matrix (m) [matrix of math nodes,row sep=1.5em,column sep=0.9em]
  { 
    0 & N & G_{c'}' & G_{c'-1}' & \cdots & G_1' & L & 0 \\
    0 & M & G_{c'} & G_{c'-1}' & \cdots & G_1' & L & 0 \\
  };
  \path[-stealth']
  (m-1-8) edge (m-1-7)
  (m-1-7) edge (m-1-6)
  (m-1-6) edge (m-1-5)
  (m-1-5) edge (m-1-4)
  (m-1-4) edge (m-1-3)
  (m-1-3) edge (m-1-2)
  (m-1-2) edge (m-1-1)
  (m-2-8) edge (m-2-7)
  (m-2-7) edge (m-2-6)
  (m-2-6) edge (m-2-5)
  (m-2-5) edge (m-2-4)
  (m-2-4) edge (m-2-3)
  (m-2-3) edge (m-2-2)
  (m-2-2) edge (m-2-1)
  (m-2-2) edge node [left] {$\phi^M_N$} (m-1-2)
  (m-2-3) edge (m-1-3)
  (m-2-4) edge [double distance=2pt,-] (m-1-4)
  (m-2-6) edge [double distance=2pt,-] (m-1-6)
  (m-2-7) edge [double distance=2pt,-] (m-1-7)
  ;
  \node at ($(m-1-1)+(-0.6,0)$) {$e^N_L:$};
  \node at ($(m-2-1)+(-0.9,0)$) {$\phi^M_N e^N_L:$};
\end{tikzpicture}
\end{center}

For $c>0$ and $c'=0$ let $e^M_N \in \Ext_\B^c(M,N)$ as above and $\psi^N_L \in \Hom_\B(N,L)$.
The \nameft{Yoneda} composite $e^M_N \psi^N_L \in \Ext_\B^c(M,L)$ is given by the pushout $c$-extension
\begin{center}
\begin{tikzpicture}
  \matrix (m) [matrix of math nodes,row sep=1.5em,column sep=0.9em]
  { 
    0 & M & G_c & \cdots & G_2 & G_1 & N & 0 \\
    0 & M & G_c & \cdots & G_2 & G_1' & L & 0 \\
  };
  \path[-stealth']
  (m-1-8) edge (m-1-7)
  (m-1-7) edge (m-1-6)
  (m-1-6) edge (m-1-5)
  (m-1-5) edge (m-1-4)
  (m-1-4) edge (m-1-3)
  (m-1-3) edge (m-1-2)
  (m-1-2) edge (m-1-1)
  (m-2-8) edge (m-2-7)
  (m-2-7) edge (m-2-6)
  (m-2-6) edge (m-2-5)
  (m-2-5) edge (m-2-4)
  (m-2-4) edge (m-2-3)
  (m-2-3) edge (m-2-2)
  (m-2-2) edge (m-2-1)
  (m-1-2) edge [double distance=2pt,-] (m-2-2)
  (m-1-3) edge [double distance=2pt,-] (m-2-3)
  (m-1-5) edge [double distance=2pt,-] (m-2-5)
  (m-1-6) edge (m-2-6)
  (m-1-7) edge node [right] {$\psi^N_L$} (m-2-7)
  ;
  \node at ($(m-1-1)+(-0.6,0)$) {$e^M_N:$};
  \node at ($(m-2-1)+(-0.9,0)$) {$e^M_N \psi^N_L:$};
\end{tikzpicture}
\end{center}

For more details see, e.g., \cite[Section~IV.9]{HS}, \cite[Appendix~B]{BB}.

\section{The binatural transformation} \label{sec:QExt}

Let $\A$ is an \nameft{Abel}ian category and $\C \subset \A$ a thick subcategory.
Applying the exact canonical functor $\QQ:\A \to \A/\C$ to a cocycle $e \in \Ext^c_\A(M,N)$ we obtain a cocycle $\QQ(e)$ in\footnote{We write $\Ext^c_{\A/\C}(M,N)$ for $\Ext^c_{\A/\C}(\QQ(M),\QQ(N))$ since $\QQ$ is the identity on objects.} $\Ext^c_{\A/\C}(M,N)$.
  In other words, the canonical functor $\QQ:\A \to \A/\C$ induces maps\footnote{By this we mean the composition $\Ext^c_\A(M',N/N') \to \Ext^c_{\A/\C}(M',N/N') \xrightarrow{\cong} \Ext^c_{\A/\C}(M,N)$, where the last isomorphism is the inverse of the one induced by the $\A$-mono $M' \hookrightarrow M$ and the $\A$-epi $N \twoheadrightarrow N/N'$, as both become isomorphisms in $\A/\C$.}
  \[
    \Ext^c_\A(M',N/N') \to \Ext^c_{\A/\C}(M,N)
  \]
  for all $M,N \in \A$, $M' \leq M$, $N' \leq N$ with $ M/M' \in \C$ and $N' \in \C$.
  For $M'' \leq M'$ with $M'/M'' \in \C$ and $N'' \geq N'$ with $N''/N' \in \C$ the cocycle
  \[
    e':\  0 \xleftarrow{} M' \xleftarrow{} G_c' \xleftarrow{} G_{c-1}' \xleftarrow{} \cdots \xleftarrow{} G_2'  \xleftarrow{} G_1' \xleftarrow{} N/N' \xleftarrow{} 0 \in \Ext^c_\A(M',N/N')
  \]
  induces a cocycle
  \[
    e'':\  0 \xleftarrow{} M'' \xleftarrow{} G_c'' \xleftarrow{} G_{c-1}' \xleftarrow{} \cdots \xleftarrow{} G_2' \xleftarrow{} G_1 '' \xleftarrow{} N/N'' \xleftarrow{} 0 \in \Ext^c_\A(M'',N/N''),
  \]
  as the \nameft{Yoneda} composite $e'' = (M'' \hookrightarrow M') e' (N/N' \twoheadrightarrow N/N'')$.
  Hence, $\QQ$ induces a map
  \[
    \QQ^{\Ext}: \varinjlim_{\substack{ M' \leq M, N' \leq N,\\ M/M' \in \C, N' \in \C }} \Ext^c_\A(M',N/N') \to \Ext^c_{\A/\C}(M,N)
  \]
  for all $M,N \in \A$.
  
  If $N \in \A$ is $\C$-saturated, then the above double limit simplifies to
  \[
    \QQ^{\Ext}: \varinjlim_{\substack{ M' \leq M,\\ M/M' \in \C}} \Ext^c_\A(M',N) \to \Ext^c_{\A/\C}(M,N)
  \]
  for all $M\in \A$, as there are no nontrivial $\C$-subobjects $N' \leq N$.

  Now we consider the functoriality of the left hand side
  \[
    F: (M,N) \mapsto \varinjlim_{\substack{ M' \leq M,\\ M/M' \in \C}} \Ext^c_\A(M',N).
  \]
  To describe the functoriality in the first argument let $\phi: M \to L$ be an $\A$-morphism, $N \in \A$ ($\C$-saturated), $G^L = G^L_N \in \Ext_\A^c(L,N)$, and $G^M = G^M_N = \phi G^L \in \Ext_\A^c(M,N)$, the \nameft{Yoneda} composition of $\phi$ and $G^L$ (by construction we have that $G^M_{c'} = G^L_{c'}$ for $c' \leq c-1$).
  Taking the pullback of a subobject $\iota_{L'}: L' \hookrightarrow L$ with $L/L' \in \C$ we obtain a subobject $\iota_M': M' \hookrightarrow M$ with $M/M' \in \C$.
  Sending the cocycle $\iota_{L'} G^L \in \Ext_\A^c(L',N)$ to $\iota_{M'} G^M \in \Ext_\A^c(M',L)$ defines the first argument action of $F$ on $\phi$.
  The proof of functoriality in the first argument follows from the identity $\iota_{M'} G^M = \iota_{M'} \phi G^L = \phi_{|M'} \iota_{L'} G^L = \phi_{|M'} G^{L'}$.
\begin{center}
\begin{tikzpicture}[=<latex,arrow/.style={preaction={draw,line width=1.2mm,white}}]
  \coordinate (right) at (2,0);
  \coordinate (up) at (0,2);
  \coordinate (behind) at (1.1,1.35);
  
  \node (M) {$M$};
  \node (L) at ($(M)+(up)$) {$L$};
  \node (M1) at ($(M)+(behind)$) {$M'$};
  \node (L1) at ($(L)+(behind)$) {$L'$};
  
  \node (GcM) at ($(M)+(right)$) {$G_c^M$};
  \node (GcL) at ($(L)+(right)$) {$G_c^L$};
  \node (GcM1) at ($(M1)+(right)$) {$G_c^{M'}$};
  \node (GcL1) at ($(L1)+(right)$) {$G_c^{L'}$};
  
  \node (Gc1M) at ($(GcM)+(right)$) {$G_{c-1}^M$};
  \node (Gc1L) at ($(GcL)+(right)$) {$G_{c-1}^L$};
  
  \node (Gc2M) at ($(Gc1M)+(right)$) {$G_{c-2}^M$};
  \node (Gc2L) at ($(Gc1L)+(right)$) {$G_{c-2}^L$};
  
  \draw[arrow,-stealth'] (Gc1M) -- (GcM1);
  \draw[arrow,-stealth'] (GcM1) -- (M1);
  
  \draw[arrow,-stealth'] (L) -- ++($-0.8*(right)$) node[left]{$0$};
  \draw[arrow,-stealth'] (M) -- ++($-0.8*(right)$) node[left]{$0$};
  \draw[arrow,-stealth'] (L1) -- ++($-0.8*(right)$) node[left]{$0$};
  \draw[arrow,-stealth'] (M1) -- ++($-0.8*(right)$) node[left]{$0$};
  
  \draw[arrow,left hook-stealth'] (M1) -- node[below right]{\scriptsize $\iota_{M'}$} (M);
  \draw[arrow,left hook-stealth'] (L1) -- node[above left]{\scriptsize $\iota_{L'}$} (L);
  \draw[arrow,stealth'-] (L) -- node[left]{\scriptsize $\phi$} (M);
  \draw[arrow,stealth'-] (L1) -- node[right,near start]{\scriptsize $\phi_{|M'}$} (M1);
  
  \draw[arrow,left hook-stealth'] (GcM1) -- (GcM);
  \draw[arrow,left hook-stealth'] (GcL1) -- (GcL);
  \draw[arrow,stealth'-] (GcL) -- (GcM);
  \draw[arrow,stealth'-] (GcL1) -- (GcM1);
  
  \draw[double distance=2pt,thick] (Gc1L) -- (Gc1M);
  \draw[double distance=2pt,thick] (Gc2L) -- (Gc2M);
  
  
  \draw[arrow,-stealth'] (Gc1L) -- (GcL1);
  \draw[arrow,-stealth'] (GcL1) -- (L1);
  \draw[arrow,-stealth'] (Gc2L) -- (Gc1L);
  \draw[arrow,-stealth'] (Gc1L) -- (GcL);
  \draw[arrow,-stealth'] (GcL) -- (L);
  \draw[arrow,-stealth'] (Gc2M) -- (Gc1M);
  \draw[arrow,-stealth'] (Gc1M) -- (GcM);
  \draw[arrow,-stealth'] (GcM) -- (M);
\end{tikzpicture}
\end{center}
  For the functoriality in the second argument consider a morphism $\psi: N \to L$ and take the colimit $\varinjlim_{\substack{ M' \leq M,\\ M/M' \in \C}}$ of the maps $\Ext^c_\A(M',N) \xrightarrow{\Ext^c_\A(M',\psi)} \Ext^c_\A(M',L)$ given by the usual functoriality of $\Ext^c_\A$ in the second argument.
  Finally, the exact functor $\QQ$ commutes with pullbacks, implying the binaturality of $\QQ^{\Ext}$.

  As $\QQ$ is exact, the map $\QQ^{\Ext}$ respects \nameft{Baer} sums.
  
\section{The proof} \label{sec:proof}

Our goal is to give sufficient conditions for the binatural transformation $\QQ^{\Ext}$ to be an isomorphism.
For this we assume that $\C\subset\A$ is a localizing subcategory of the \nameft{Abel}ian category $\A$.
Then the restricted canonical functor $\QQ:\Sat_\C(\A) \to \A/\C$ and the corestricted section functor $\SS:\A/\C \to \Sat_\C(\A)$ are adjoint equivalences of categories.

\begin{rmrk} \label{rmrk:saturated_form}
We will use this equivalence to replace $\Ext_{\A/\C}$ by the isomorphic $\Ext_{\Sat_\C(\A)}$, the functor $\QQ: \A \to \A/\C$ by $\widehat{\QQ}:=\cores_{\Sat_\C(\A)} (\SS \circ \QQ): \A \to \Sat_\C(\A)$, and finally $\QQ^{\Ext}$ by
\[
  \widehat{\QQ}^{\Ext}: \varinjlim_{\substack{ M' \leq M,\\ M/M' \in \C }} \Ext^c_\A(M',N) \to \Ext^c_{\Sat_\C(\A)}(M,N) \mbox{.}
\]
For simplicity we write $\Ext^c_{\Sat_\C(\A)}(M,N)$ for $\Ext^c_{\Sat_\C(\A)}(\widehat{\QQ}(M),\widehat{\QQ}(N))$.
Recall that in Theorem~\ref{thm:main} we require $M$ to be $\C$-torsion-free and $N$ to be $\C$-saturated.
Since the cokernel $(\SS \circ \QQ)(M) / M$ of $\eta_M$ lies in $\C$ we can, without loss of generality, as well assume $M$ to be $\C$-saturated as the limit does not distinguish between $M$ and its saturation $(\SS \circ \QQ)(M)$.
\end{rmrk}

\subsection{The proof for \texorpdfstring{$\Ext^1$}{Ext1}}

For $c=1$ it turns out that assuming $\C\subset \A$ to be localizing is already sufficient for $\widehat{\QQ}^{\Ext}$ to be an isomorphism.
\begin{theorem} \label{thm:ext1}
  If $\C$ is a localizing subcategory of the \nameft{Abel}ian category $\A$ then
  \[
  \widehat{\QQ}^{\Ext}: \varinjlim_{\substack{ M' \leq M,\\ M/M' \in \C }} \Ext^1_\A(M',N) \to \Ext^1_{\Sat_\C(\A)}(M,N) \cong \Ext^1_{\A/\C}(M,N)
  \]
  is an isomorphism (of \nameft{Abel}ian groups) for all $\C$-saturated $M,N \in \A$.
\end{theorem}
\begin{proof}
  
  Recall that a short exact sequence $e:\  0 \xleftarrow{} M \xleftarrow{\pi} E \xleftarrow{} N \xleftarrow{} 0$ in $\Sat_\C(\A)$ is in general only left exact in $\A$, since the embedding functor $\iota: \Sat_\C(\A) \hookrightarrow \A$ is in general only left exact.
  The $\A$-cokernel of $\pi$ lies in $\C$, i.e., for $M' := \img \pi$ the sequence $0 \xleftarrow{} M' \xleftarrow{\pi} E \xleftarrow{} N \xleftarrow{} 0$ is exact in $\A$ and $M/M' = \coker \pi \in \C$.
  This yields the preimage of $e$ under $\widehat{\QQ}^{\Ext}$ and shows surjectivity.
  
  For the injectivity take an exact $\A$-sequence $e:\  0 \xleftarrow{} M' \xleftarrow{} E \xleftarrow{\phi} N \xleftarrow{} 0$ such that the corresponding exact $\Sat_\C(\A)$-sequence
  \[
    \widehat{\QQ}^{\Ext}(e):\  0 \xleftarrow{} \widehat{\QQ}(M') \xleftarrow{} \widehat{\QQ}(E) \xleftarrow{\widehat{\QQ}(\phi)} \widehat{\QQ}(N) \xleftarrow{} 0
  \]
  is split, i.e., $e$ is in the kernel of $\widehat{\QQ}^{\Ext}$.
  By definition of split short exact sequences, there is a $\widehat{\psi}: \widehat{\QQ}(E)\to \widehat{\QQ}(N)$ such that $\widehat{\psi}\circ\widehat{\QQ}(\phi)=\Id_{\widehat{\QQ}(N)}$.
  Since $N$ is $\C$-saturated the unit $\eta_N: N \to \iota(\widehat{\QQ}(N))$ is an isomorphism and we can define $\psi:=\eta_N^{-1}\circ \iota(\widehat{\psi}) \circ \eta_E$.
  Note that $\eta_E \circ \phi = \iota(\widehat{\QQ}(\phi)) \circ \eta_N$, by the naturality of $\eta$.
  Then $\psi\circ\phi=\eta_N^{-1}\circ \iota(\widehat{\psi}) \circ \eta_E \circ \phi=\eta_N^{-1}\circ \iota(\widehat{\psi}) \circ \iota(\widehat{\QQ}(\phi)) \circ \eta_N=\eta_N^{-1}\circ \iota(\Id_{\widehat{\QQ}(N)}) \circ \eta_N=\eta_N^{-1}\circ \Id_{\iota(\widehat{\QQ}(N))} \circ \eta_N=\Id_N$ implies that $e$ is split, i.e., zero in $\Ext^1_\A(M',N)$.
\end{proof}

\subsection{The proof of surjectivity for higher \texorpdfstring{$\Ext$}{Ext}'s}

For $c \geq 2$ we need further conditions on the categories $\A$ and $\C$.

\begin{defn} \label{defn:almost_split}
  Let $\A$ be an \nameft{Abel}ian category and $\C \subset \A$ a thick subcategory.
  For an object $a \in \A$ we call a subobject $a^\perp \leq a$ an \textbf{almost $\C$-complement} if $a^\perp$ is $\C$-torsion-free and $a / a^\perp \in \C$.
  We call $\C$ an \textbf{almost split} (thick) subcategory if for each object $a \in \A$ there exists an almost $\C$-complement $a^\perp$.
\end{defn}

\begin{rmrk} \label{rmrk:cnd_in_noeth_case}
  Let $\A$ be an \nameft{Abel}ian \nameft{Noether}ian category and $\C \subset \A$ a thick subcategory.
  The following two properties are equivalent:
  \begin{enumerate}
    \item $\C$ is almost split.
    \item For each object $a \in \A$ which does not lie in $\C$ there exists a nontrivial $\C$-torsion-free subobject of $a$. \label{rmrk:cnd_in_noeth_case.b}
  \end{enumerate}
\end{rmrk}
\begin{proof}
  We only discuss the nontrivial direction.
  Start with a nontrivial $\C$-torsion-free subobject $a_1 \leq a$.
  If $a / a_1$ lies in $\C$ we are done.
  Otherwise define $a_2$ to be the preimage in $a$ of a nontrivial $\C$-torsion-free subobject in $a / a_1$.
  By Lemma~\ref{lemm:extension_of_C_free_is_C_free}, $a_2$ is $\C$-torsion-free.
  Iterating the process yields a strictly ascending chain of $\C$-torsion-free subobjects of $a$.
  Due to \nameft{Noether}ianity this iteration has to stop, say at $a_n$, and it can only stop at $a_n$ if $a / a_n$ lies in $\C$.
\end{proof}

\begin{exmp} \label{exmp:almost_split}
  Let $S$ be the polynomial ring $k[x_0,\ldots,x_n]$ graded by total degree and $\A$ the category of f.g.~graded $S$-module.
  Consider the thick subcategory $\C$ of $0$-dimensional graded modules.
  These are the modules living in a finite degree interval.
  For any $M \in \A$ there exists a maximal submodule $N \in \C$, and let $d$ be the smallest integer with $N_{\geq d} = 0$.
  Then $M_{\geq d}$ is a nontrivial $\C$-torsion-free submodule of $M$.
  Recall that $\A / \C \simeq \Coh \PP^n_k$.
\end{exmp}

\begin{exmp} \label{exmp:non_almost_split}
  Consider the cone $\sigma = \operatorname{Cone}( (1,0), (1,2) ) \subset \mathbb{R}^2$ and the nonsmooth affine toric variety $U_\sigma$ with \nameft{Cox} ring $S=k[x,y]$, $\deg x=\deg y=1 \in \operatorname{Cl} (U_\sigma) = \Z/2\Z$, and affine coordinate ring $S_0=k[x^2,xy,y^2]$.
  The kernel of the sheafification from the category $\A$ of f.g.~graded $S$-modules to $\Coh U_\sigma$ is the thick subcategory $\C$ of graded modules with $M_0 = 0$ (cf.~\cite[Proposition~5.3.3]{CLS11}).
  The graded $S$-module $M := S / \langle x^2,xy,y^2 \rangle$ violates condition \eqref{rmrk:cnd_in_noeth_case.b} of the previous remark.
  Hence, $\C \subset \A$ is not almost split.
  
  Now we show that $\widehat{\QQ}^{\Ext}$ is not surjective.
  Note that all $\Ext^c_\A$ vanish for $c>2$, where $2$ is the global dimension of $S$.
  However, for the $S_0$-module $k = S_0 / \langle x^2,xy,y^2 \rangle$ the group $\Ext^c_{\A/\C}(k,k) \neq 0$ for all $c \in \Z_{\geq 0}$ (in fact, $\Ext^c_{\A/\C}(k,k) \cong k^4$ for all $c>1$).
  The sheafification of $k$ is the skyscraper sheaf in $\Coh U_\sigma \simeq \A/\C$ on the singular point of $\Spec(S_0)$.
\end{exmp}

One can replace $\C$-torsion-free $\A$-complexes having defects in $\C$ with exact $\A$-complexes, which are equivalent in the following sense:
\begin{defn}
  Let $\C$ be a thick subcategory of the \nameft{Abel}ian category $\A$ and $e$ an $\A$-complex.
  We say a subcomplex $e'$ \textbf{equals $e$ up to $\C$-factors} if $e/e'$ is a complex in $\C$.
\end{defn}

\begin{lemm} \label{lemm:Cfree_exact}
  Let $\C$ be an almost split thick subcategory of the \nameft{Abel}ian category $\A$ and
  \[
    e:\  0 \xleftarrow{} M \xleftarrow{} G_c \xleftarrow{} \cdots \xleftarrow{} G_1 \xleftarrow{} N \xleftarrow{} 0\mbox{.}
  \]
  a $\C$-torsion-free $\A$-complex which is exact up to $\C$-defects.
  Then there exists an \emph{exact} ($\C$-torsion-free) $\A$-subcomplex $e^\perp$ of $e$
  \[
    e^\perp:\  0 \xleftarrow{} M^\perp \xleftarrow{} G_c^\perp \xleftarrow{} \cdots \xleftarrow{} G^\perp_1 \xleftarrow{} N \xleftarrow{} 0
  \]
  which equals $e$ up to $\C$-factors.
\end{lemm}
\begin{center}
 \begin{tikzpicture}[
  C/.style={color=red,line width=1.5pt,dotted},
  new/.style={color=blue,line width=2pt},
  map/.style={color=gray},
  filled/.style={line width=5,color=white},
  scale=0.84]
    
  \newcommand{\circlewidth}{1.5pt}
  
  \coordinate (a) at (0,1.0);
  \coordinate (c) at (-0.5,0.4);
  \coordinate (m) at (-1.5,0);
  
  \coordinate (G0b);
  \coordinate (G0) at ($(G0b)+(a)$);
  
  \coordinate (0_G1) at ($(G0b)+(m)$);
  \coordinate (G1b) at ($(G0)+(m)$);
  \coordinate (G1c) at ($(G1b)+(c)$);
  \coordinate (G1d) at ($(G1b)+(a)$);
  \coordinate (G1cd) at ($(G1c)+(a)$);
  \coordinate (G1) at ($(G1cd)+(c)$);
  
  \coordinate (0_G2) at ($(G1c)+(m)$);
  \coordinate (G2b) at ($(G1cd)+(m)$);
  \coordinate (G2c) at ($(G1)+(m)$);
  \coordinate (G2d) at ($(G2b)+(a)$);
  \coordinate (G2cd) at ($(G2c)+(a)$);
  \coordinate (G2cc) at ($(G2c)+(c)$);
  \coordinate (G2ccd) at ($(G2cc)+(a)$);
  \coordinate (G2) at ($(G2ccd)+(c)$);
  
  \coordinate (0_Gc1) at ($(G2cc)+3*(m)$);
  \coordinate (Gc1b) at ($(G2ccd)+3*(m)$);
  \coordinate (Gc1c) at ($(G2)+3*(m)$);
  \coordinate (Gc1d) at ($(Gc1b)+(a)$);
  \coordinate (Gc1cd) at ($(Gc1c)+(a)$);
  \coordinate (Gc1cc) at ($(Gc1c)+(c)$);
  \coordinate (Gc1ccd) at ($(Gc1cc)+(a)$);
  \coordinate (Gc1) at ($(Gc1ccd)+(c)$);
  
  \coordinate (0_M) at ($(Gc1cc)+(m)$);
  \coordinate (Mb) at ($(Gc1ccd)+(m)$);
  \coordinate (Mc) at ($(Gc1)+(m)$);
  \coordinate (M) at ($(Mc)+(c)$);
  
  \node at (G0b) [left] {$ $};
  \node at (G0) [above] {$N$};
  
  \node at (0_G1) [left] {$ $};
  \node at (G1b) [left] {$ $};
  \node at (G1c) [left] {$ $};
  \node at (G1d) [above right] {$G_1^\perp$};
  \node at (G1cd) [left] {$ $};
  \node at (G1) [above right] {$G_1$};
  
  \node at (0_G2) [left] {$ $};
  \node at (G2b) [left] {$ $};
  \node at (G2c) [left] {$ $};
  \node at (G2d) [above right] {$G_2^\perp$};
  \node at (G2cd) [left] {$ $};
  \node at (G2cc) [left] {$ $};
  \node at (G2ccd) [left] {$ $};
  \node at (G2) [above right] {$G_2$};
  
  \node at (0_Gc1) [left] {$ $};
  \node at (Gc1b) [left] {$ $};
  \node at (Gc1c) [left] {$ $};
  \node at (Gc1d) [above right] {$G_c^\perp$};
  \node at (Gc1cd) [left] {$ $};
  \node at (Gc1cc) [left] {$ $};
  \node at (Gc1ccd) [left] {$ $};
  \node at (Gc1) [above right] {$G_c$};
  
  \node at (0_M) [left] {$ $};
  \node at (Mb) [below left=-1pt,yshift=3pt] {$M^\perp$};
  \node at (Mc) [above,xshift=3pt] {$\widetilde{M}$};
  \node at (M) [left] {$M$};
  
  \filldraw[filled,fill=green!25] \FillArea{(G0b)}{(G0)}{(G1b)}{(0_G1)};
  
  \filldraw[filled,fill=yellow!35] \FillArea{(G1b)}{(G1d)}{(G1cd)}{(G1c)};
  \filldraw[filled,fill=yellow!35] \FillArea{(G1cd)}{(G1c)}{(0_G2)}{(G2b)};
  
  \filldraw[filled,fill=orange!25] \FillArea{(G2b)}{(G2d)}{(G2cd)}{(G2c)};
  \filldraw[filled,fill=orange!25] \FillArea{(G2c)}{(G2cd)}{(G2ccd)}{(G2cc)};
  \filldraw[filled,fill=orange!25] \FillArea{(G2ccd)}{(G2cc)}{($(G2cc)+(m)$)}{($(G2ccd)+(m)$)};
  
  \filldraw[filled,fill=red!25] \FillArea{(0_Gc1)}{(Gc1b)}{($(Gc1b)-(m)$)}{($(0_Gc1)-(m)$)};
  
  \filldraw[filled,fill=purple!40] \FillArea{(Gc1b)}{(Gc1d)}{(Gc1cd)}{(Gc1c)};
  \filldraw[filled,fill=purple!40] \FillArea{(Gc1c)}{(Gc1cd)}{(Gc1ccd)}{(Gc1cc)};
  \filldraw[filled,fill=purple!40] \FillArea{(Gc1ccd)}{(Gc1cc)}{(0_M)}{(Mb)};
  
  \draw[new] (G0b) -- (G0);
  
  \draw[new] (0_G1) -- (G1b)  -- (G1d);
  \draw[C] (G1b) -- (G1c);
  \draw[new] (G1c) -- (G1cd);
  \draw[C] (G1d) -- (G1cd) -- (G1);
  
  \draw[new] (0_G2) -- (G2b) -- (G2d);
  \draw[new] (G2c) -- (G2cd);
  \draw[new] (G2cc) -- (G2ccd);
  \draw[C] (G2b) -- (G2c) -- (G2cc);
  \draw[C] (G2d) -- (G2cd) -- (G2ccd) -- (G2);
  
  \draw[new] (0_Gc1) -- (Gc1b) -- (Gc1d);
  \draw[new] (Gc1c) -- (Gc1cd);
  \draw[new] (Gc1cc) -- (Gc1ccd);
  \draw[C] (Gc1b) -- (Gc1c) -- (Gc1cc);
  \draw[C] (Gc1d) -- (Gc1cd) -- (Gc1ccd) -- (Gc1);
  
  \draw[new] (0_M) -- (Mb);
  \draw[C] (Mb) -- (Mc) -- (M);
  
  \fill (G0b) circle (\circlewidth);
  \fill (G0) circle (\circlewidth);
  
  \fill (0_G1) circle (\circlewidth);
  \fill (G1b) circle (\circlewidth);
  \fill (G1c) circle (\circlewidth);
  \fill (G1d) circle (\circlewidth);
  \fill (G1cd) circle (\circlewidth);
  \fill (G1) circle (\circlewidth);
  
  \fill (0_G2) circle (\circlewidth);
  \fill (G2b) circle (\circlewidth);
  \fill (G2c) circle (\circlewidth);
  \fill (G2d) circle (\circlewidth);
  \fill (G2cd) circle (\circlewidth);
  \fill (G2cc) circle (\circlewidth);
  \fill (G2ccd) circle (\circlewidth);
  \fill (G2) circle (\circlewidth);
  
  \fill (0_Gc1) circle (\circlewidth);
  \fill (Gc1b) circle (\circlewidth);
  \fill (Gc1c) circle (\circlewidth);
  \fill (Gc1d) circle (\circlewidth);
  \fill (Gc1cd) circle (\circlewidth);
  \fill (Gc1cc) circle (\circlewidth);
  \fill (Gc1ccd) circle (\circlewidth);
  \fill (Gc1) circle (\circlewidth);
  
  \fill (0_M) circle (\circlewidth);
  \fill (Mb) circle (\circlewidth);
  \fill (Mc) circle (\circlewidth);
  \fill (M) circle (\circlewidth);
  
  
  \coordinate (right) at (0.1,0);
  \coordinate (left) at (-0.1,0);
  
  \draw[map,stealth'-] ($(0_G1)+(right)$) -- ($(left)+(G0b)$);
  \draw[map,stealth'-] ($(G1b)+(right)$) -- ($(left)+(G0)$);
  
  \draw[map,stealth'-] ($(0_G2)+(right)$) -- ($(left)+(G1c)$);
  \draw[map,stealth'-] ($(G2b)+(right)$) -- ($(left)+(G1cd)$);
  \draw[map,stealth'-] ($(G2c)+(right)$) -- ($(left)+(G1)$);
  
  \draw[map,stealth'-] ($(0_Gc1)+(right)$) -- ($(0_Gc1)-(right)-(m)$);
  \draw[map,stealth'-] ($(Gc1b)+(right)$) -- ($(Gc1b)-(right)-(m)$);
  \draw[map,stealth'-] ($(Gc1c)+(right)$) -- ($(Gc1c)-(right)-(m)$);
  \node at ($1/2*(0_Gc1)+1/2*(G2)$) {$\dots$};
  \draw[map,stealth'-] ($(G2cc)+(m)-(left)$) -- ($(left)+(G2cc)$);
  \draw[map,stealth'-] ($(G2ccd)+(m)-(left)$) -- ($(left)+(G2ccd)$);
  \draw[map,stealth'-] ($(G2)+(m)-(left)$) -- ($(left)+(G2)$);
  
  \draw[map,stealth'-] ($(0_M)+(right)$) -- ($(left)+(Gc1cc)$);
  \draw[map,stealth'-] ($(Mb)+(right)$) -- ($(left)+(Gc1ccd)$);
  \draw[map,stealth'-] ($(Mc)+(right)$) -- ($(left)+(Gc1)$);
  
 \end{tikzpicture}
\end{center}
\begin{proof}
  Define $\overline{G}_1 := G_1/\img_\A(G_1 \hookleftarrow N)$.
  Construct the subobject $G_1^\perp \leq G_1$ as the preimage in $G_1$ of an almost $\C$-complement in $\overline{G}_1$.
  Since $G_1^\perp$ is the preimage of an almost $\C$-complement it follows that $G_1/G_1^\perp \in \C$.
  
  For $i>1$ we assume to have constructed $G_{i-1}^\perp \le G_{i-1}$ with $G_{i-1}/G_{i-1}^\perp \in \C$.
  We proceed inductively and consider $\overline{G}_i := G_i/\img_\A(G_i \xleftarrow{} G_{i-1} \hookleftarrow G_{i-1}^\perp)$.
  Again construct the subobject $G_i^\perp \leq G_i$ as the preimage in $G_i$ of an almost $\C$-complement in $\overline{G}_i$.
  As above $G_i/G_i^\perp \in \C$.
  
  Finally define the subobject $M^\perp \leq M$ as the $\A$-image $\img_\A(M \xleftarrow{} G_{c-1} \hookleftarrow G_{c-1}^\perp)$.
  Let $\widetilde{M} := \img_\A(M \xleftarrow{} G_{c-1})$.
  Then $\widetilde{M}/M^\perp \in \C$ as an epimorphic image of $G_{c-1}/G_{c-1}^\perp$ under $M \xleftarrow{} G_{c-1}$.
  Since also $M/\widetilde{M} \in \C$ it follows that $M/M^\perp \in \C$ as an extension of two objects in $\C$.
  The whole argument is visualized in the diagram\footnote{Cf.~\cite{BaSF} for the use of \nameft{Hasse} diagrams to prove statements in \nameft{Abel}ian categories.} below, where the dotted lines stand for (factor) objects in $\C$.
\end{proof}

The above lemma yields the preimages needed to prove the surjectivity of $\widehat{\QQ}^{\Ext}$.

\begin{prop} \label{prop:surj}
  Let $\C$ be an almost split localizing subcategory of the \nameft{Abel}ian category $\A$.
  Then
  \[
    \widehat{\QQ}^{\Ext}: \varinjlim_{\substack{ M' \leq M,\\ M/M' \in \C}} \Ext^c_\A(M',N) \to \Ext^c_{\Sat_\C(\A)}(M,N)
  \]
  is an epimorphism (of \nameft{Abel}ian groups) for all $\C$-saturated $M,N \in \A$.
 \end{prop}

\begin{proof}
  For the surjectivity consider a $c$-extension $\widehat{e} \in \Ext^c_{\Sat_\C(\A)}(M,N)$ for $c > 0$, represented by an exact $\Sat_\C(\A)$-complex
  \[
    \widehat{e}:\  0 \xleftarrow{} M \xleftarrow{} G_c \xleftarrow{} \cdots \xleftarrow{} G_1 \xleftarrow{} N \xleftarrow{} 0\mbox{.}
  \]
  Lemma~\ref{lemm:Cfree_exact} applied to the $\A$-complex $e = \iota(\widehat{e})$ which is exact up to $\C$-defects yields a preimage $e^\perp$ of $\widehat{e}$.
\end{proof}

Due to the left exactness of $\iota$ we can even choose $G_1^\perp:=G_1$ when applying Lemma~\ref{lemm:Cfree_exact} in the proof of Proposition~\ref{prop:surj}.
This is illustrated by the diagram below.
\begin{center}
 \begin{tikzpicture}[
  C/.style={color=red,line width=1.5pt,dotted},
  new/.style={color=blue,line width=2pt},
  map/.style={color=gray},
  filled/.style={line width=4,color=white},
  scale=0.6]
    
  \newcommand{\circlewidth}{1.5pt}
  
  \coordinate (a) at (0,1.0);
  \coordinate (c) at (-0.5,0.4);
  \coordinate (m) at (-1.7,0);
  
  \coordinate (0_N);
  \coordinate (N) at ($(0_N)+(a)$);
  
  \coordinate (0_G0) at ($(0_N)+(m)$);
  \coordinate (G0b) at ($(N)+(m)$);
  \coordinate (G0) at ($(G0b)+(a)$);
  
  \coordinate (0_G1) at ($(G0b)+(m)$);
  \coordinate (G1b) at ($(G0)+(m)$);
  \coordinate (G1c) at ($(G1b)+(c)$);
  \coordinate (G1d) at ($(G1b)+(a)$);
  \coordinate (G1cd) at ($(G1c)+(a)$);
  \coordinate (G1) at ($(G1cd)+(c)$);
  
  \coordinate (0_G2) at ($(G1c)+(m)$);
  \coordinate (G2b) at ($(G1cd)+(m)$);
  \coordinate (G2c) at ($(G1)+(m)$);
  \coordinate (G2d) at ($(G2b)+(a)$);
  \coordinate (G2cd) at ($(G2c)+(a)$);
  \coordinate (G2cc) at ($(G2c)+(c)$);
  \coordinate (G2ccd) at ($(G2cc)+(a)$);
  \coordinate (G2) at ($(G2ccd)+(c)$);
  
  \coordinate (0_Gc1) at ($(G2cc)+3*(m)$);
  \coordinate (Gc1b) at ($(G2ccd)+3*(m)$);
  \coordinate (Gc1c) at ($(G2)+3*(m)$);
  \coordinate (Gc1d) at ($(Gc1b)+(a)$);
  \coordinate (Gc1cd) at ($(Gc1c)+(a)$);
  \coordinate (Gc1cc) at ($(Gc1c)+(c)$);
  \coordinate (Gc1ccd) at ($(Gc1cc)+(a)$);
  \coordinate (Gc1) at ($(Gc1ccd)+(c)$);
  
  \coordinate (0_M) at ($(Gc1cc)+(m)$);
  \coordinate (Mb) at ($(Gc1ccd)+(m)$);
  \coordinate (Mc) at ($(Gc1)+(m)$);
  \coordinate (M) at ($(Mc)+(c)$);
  
  \node at (0_N) [left] {$ $};
  \node at (N) [above] {$N$};
  
  \node at (0_G0) [left] {$ $};
  \node at (G0b) [left] {$ $};
  \node at (G0) [above right] {$G_1=G_1^\perp$};
  
  \node at (0_G1) [left] {$ $};
  \node at (G1b) [left] {$ $};
  \node at (G1c) [left] {$ $};
  \node at (G1d) [right] {$G_2^\perp$};
  \node at (G1cd) [left] {$ $};
  \node at (G1) [above right] {$G_2$};
  
  \node at (0_G2) [left] {$ $};
  \node at (G2b) [left] {$ $};
  \node at (G2c) [left] {$ $};
  \node at (G2d) [above right] {$G_3^\perp$};
  \node at (G2cd) [left] {$ $};
  \node at (G2cc) [left] {$ $};
  \node at (G2ccd) [left] {$ $};
  \node at (G2) [above right] {$G_3$};
  
  \filldraw[filled,fill=blue!20] \FillArea{(0_N)}{(N)}{(G0b)}{(0_G0)};
  
  \filldraw[filled,fill=green!25] \FillArea{(G0b)}{(G0)}{(G1b)}{(0_G1)};
  
  \filldraw[filled,fill=yellow!35] \FillArea{(G1b)}{(G1d)}{(G1cd)}{(G1c)};
  \filldraw[filled,fill=yellow!35] \FillArea{(G1cd)}{(G1c)}{(0_G2)}{(G2b)};
  
  \filldraw[filled,fill=orange!25] \FillArea{(G2b)}{(G2d)}{(G2cd)}{(G2c)};
  \filldraw[filled,fill=orange!25] \FillArea{(G2c)}{(G2cd)}{(G2ccd)}{(G2cc)};
  \filldraw[filled,fill=orange!25] \FillArea{(G2ccd)}{(G2cc)}{($(G2cc)+(m)$)}{($(G2ccd)+(m)$)};
  
  \draw[new] (0_N) -- (N);
  
  \draw[new] (0_G0) -- (G0b) -- (G0);
  
  \draw[new] (0_G1) -- (G1b)  -- (G1d);
  \draw[C] (G1b) -- (G1c);
  \draw (G1c) -- (G1cd);
  \draw[C] (G1d) -- (G1cd) -- (G1);
  
  \draw[new] (0_G2) -- (G2b) -- (G2d);
  \draw (G2c) -- (G2cd);
  \draw (G2cc) -- (G2ccd);
  \draw[C] (G2b) -- (G2c) -- (G2cc);
  \draw[C] (G2d) -- (G2cd) -- (G2ccd) -- (G2);
  
  \fill (0_N) circle (\circlewidth);
  \fill (N) circle (\circlewidth);
  
  \fill (0_G0) circle (\circlewidth);
  \fill (G0b) circle (\circlewidth);
  \fill (G0) circle (\circlewidth);
  
  \fill (0_G1) circle (\circlewidth);
  \fill (G1b) circle (\circlewidth);
  \fill (G1c) circle (\circlewidth);
  \fill (G1d) circle (\circlewidth);
  \fill (G1cd) circle (\circlewidth);
  \fill (G1) circle (\circlewidth);
  
  \fill (0_G2) circle (\circlewidth);
  \fill (G2b) circle (\circlewidth);
  \fill (G2c) circle (\circlewidth);
  \fill (G2d) circle (\circlewidth);
  \fill (G2cd) circle (\circlewidth);
  \fill (G2cc) circle (\circlewidth);
  \fill (G2ccd) circle (\circlewidth);
  \fill (G2) circle (\circlewidth);
  
  
  \coordinate (right) at (0.1,0);
  \coordinate (left) at (-0.1,0);
  
  \draw[map,stealth'-] ($(0_G0)+(right)$) -- ($(left)+(0_N)$);
  \draw[map,stealth'-] ($(G0b)+(right)$) -- ($(left)+(N)$);
  
  \draw[map,stealth'-] ($(0_G1)+(right)$) -- ($(left)+(G0b)$);
  \draw[map,stealth'-] ($(G1b)+(right)$) -- ($(left)+(G0)$);
  
  \draw[map,stealth'-] ($(0_G2)+(right)$) -- ($(left)+(G1c)$);
  \draw[map,stealth'-] ($(G2b)+(right)$) -- ($(left)+(G1cd)$);
  \draw[map,stealth'-] ($(G2c)+(right)$) -- ($(left)+(G1)$);
  \node at ($1/2*(0_Gc1)+1/2*(G2)$) {$\dots$};
  \draw[map,stealth'-] ($(G2cc)+(m)-(left)$) -- ($(left)+(G2cc)$);
  \draw[map,stealth'-] ($(G2ccd)+(m)-(left)$) -- ($(left)+(G2ccd)$);
  \draw[map,stealth'-] ($(G2)+(m)-(left)$) -- ($(left)+(G2)$);

 \end{tikzpicture}
\end{center}

\subsection{The proof of injectivity for higher \texorpdfstring{$\Ext$}{Ext}'s}

To prove the injectivity we show that almost $\C$-complements exist on the level of exact $\A$-complexes.

\begin{defn}
  Let $\C$ be a thick subcategory of the \nameft{Abel}ian category $\A$ and
  \[
    e:\  0 \xleftarrow{} M \xleftarrow{} G_c \xleftarrow{} \cdots \xleftarrow{} G_1 \xleftarrow{} N \xleftarrow{} 0
  \]
  an $\A$-complex where $M, N$ are $\C$-torsion-free.
  We call a $\C$-torsion-free $\A$-subcomplex $\widetilde{e} \leq e$
  \[
    \widetilde{e}:\  0 \xleftarrow{} \widetilde{M} \xleftarrow{} \widetilde{G}_c \xleftarrow{} \cdots \xleftarrow{} \widetilde{G}_1 \xleftarrow{} N \xleftarrow{} 0\mbox{,}
  \]
  which equals $e$ up to $\C$-factors an \textbf{almost $\C$-complement} in $e$.
\end{defn}

\begin{prop} \label{prop:Cfree}
  Let $\C$ be an almost split thick subcategory of the \nameft{Abel}ian category $\A$ and
  \[
    e:\  0 \xleftarrow{} M \xleftarrow{} G_c \xleftarrow{} \cdots \xleftarrow{} G_1 \xleftarrow{} N \xleftarrow{} 0
  \]
  an \emph{exact} $\A$-complex where $M, N$ are $\C$-torsion-free.
  Then there exists an \emph{exact} almost $\C$-complement $\widetilde{e} \in \Ext^c_\A(\widetilde{M},N)$ in $e$.
\end{prop}

The value of this proposition lies in the following fact:
\begin{lemm}
  The $\A$-subcomplex $\widetilde{e} \leq e$ in the previous proposition represents in the colimit $\varinjlim_{\substack{ M' \leq M,\\ M/M' \in \C}} \Ext^c_\A(M',N)$ the same $c$-cocycle as $e$.
\end{lemm}
\begin{proof}
  The cocycle $e \in \Ext^c_\A(M,N)$ is identical to the \nameft{Yoneda} product
  \[
    e' := (\widetilde{M} \hookrightarrow M) e:\ 0 \xleftarrow{} \widetilde{M} \xleftarrow{} G_c' \xleftarrow{} G_{c-1} \xleftarrow{} \cdots \xleftarrow{} G_1 \xleftarrow{} N \xleftarrow{} 0\in \Ext^c_\A(\widetilde{M},N),
  \]
  in the colimit $\varinjlim_{\substack{ M' \leq M,\\ M/M' \in \C}} \Ext^c_\A(M',N)$.
\begin{center}
\begin{tikzpicture}
  \coordinate (r) at (1.3,0);
  \coordinate (d) at (0,-1.3);
  
  \node (e) {$e:$};
  \node (0e) at ($(e)+0.5*(r)$) {$0$};
  \node (M) at ($(0e)+(r)$) {$M$};
  \node (Gc) at ($(M)+(r)$) {$G_c$};
  \node (Gc_1) at ($(Gc)+(r)$) {$G_{c-1}$};
  \node (edots) at ($(Gc_1)+(r)$) {$\cdots$};
  \node (G2) at ($(edots)+(r)$) {$G_2$};
  \node (G1) at ($(G2)+(r)$) {$G_1$};
  \node (N) at ($(G1)+(r)$) {$N$};
  \node (e0) at ($(N)+(r)$) {$0$};
  
  \draw[-stealth'] (M) -- (0e);
  \draw[-stealth'] (Gc) -- (M);
  \draw[-stealth'] (Gc_1) -- (Gc);
  \draw[-stealth'] (edots) -- (Gc_1);
  \draw[-stealth'] (G2) -- (edots);
  \draw[-stealth'] (G1) -- (G2);
  \draw[-stealth'] (N) -- (G1);
  \draw[-stealth'] (e0) -- (N);
  
  \node (e') at ($(e)+(d)$) {$e':$};
  \node (0e') at ($(e')+0.5*(r)$) {$0$};
  \node (Mt) at ($(0e')+(r)$) {$\widetilde{M}$};
  \node (Gc') at ($(Mt)+(r)$) {$G_c'$};
  \node (Gc_1') at ($(Gc')+(r)$) {$G_{c-1}$};
  \node (edots') at ($(Gc_1')+(r)$) {$\cdots$};
  \node (G2') at ($(edots')+(r)$) {$G_2$};
  \node (G1') at ($(G2')+(r)$) {$G_1$};
  \node (N') at ($(G1')+(r)$) {$N$};
  \node (e'0) at ($(N')+(r)$) {$0$};
  
  \draw[-stealth'] (Mt) -- (0e');
  \draw[-stealth'] (Gc') -- (Mt);
  \draw[-stealth'] (Gc_1') -- (Gc');
  \draw[-stealth'] (edots') -- (Gc_1');
  \draw[-stealth'] (G2') -- (edots');
  \draw[-stealth'] (G1') -- (G2');
  \draw[-stealth'] (N') -- (G1');
  \draw[-stealth'] (e'0) -- (N');
  
  \draw[right hook-stealth'] (Mt) -- (M);
  \draw[right hook-stealth'] (Gc') -- (Gc);
  \draw[double distance=0.2em] (Gc_1') -- (Gc_1);
  \draw[double distance=0.2em] (G2') -- (G2);
  \draw[double distance=0.2em] (G1') -- (G1);
  \draw[double distance=0.2em] (N') -- (N);
  
  \node (et) at ($(e')+(d)$) {$\widetilde{e}:$};
  \node (0et) at ($(et)+0.5*(r)$) {$0$};
  \node (Mtt) at ($(0et)+(r)$) {$\widetilde{M}$};
  \node (Gct) at ($(Mtt)+(r)$) {$\widetilde{G}_c$};
  \node (Gc_1t) at ($(Gct)+(r)$) {$\widetilde{G}_{c-1}$};
  \node (edotst) at ($(Gc_1t)+(r)$) {$\cdots$};
  \node (G2t) at ($(edotst)+(r)$) {$\widetilde{G}_2$};
  \node (G1t) at ($(G2t)+(r)$) {$\widetilde{G}_1$};
  \node (Nt) at ($(G1t)+(r)$) {$N$};
  \node (et0) at ($(Nt)+(r)$) {$0$};
  
  \draw[-stealth'] (Mtt) -- (0et);
  \draw[-stealth'] (Gct) -- (Mtt);
  \draw[-stealth'] (Gc_1t) -- (Gct);
  \draw[-stealth'] (edotst) -- (Gc_1t);
  \draw[-stealth'] (G2t) -- (edotst);
  \draw[-stealth'] (G1t) -- (G2t);
  \draw[-stealth'] (Nt) -- (G1t);
  \draw[-stealth'] (et0) -- (Nt);
  
  \draw[double distance=0.2em] (Mtt) -- (Mt);
  \draw[right hook-stealth'] (Gct) -- (Gc');
  \draw[right hook-stealth'] (Gc_1t) -- (Gc_1');
  \draw[right hook-stealth'] (G2t) -- (G2');
  \draw[right hook-stealth'] (G1t) -- (G1');
  \draw[double distance=0.2em] (Nt) -- (N');
  
  \draw[right hook-stealth'] (Gct) to [bend right=35] (Gc);
  \draw[-doublestealth] (Gct) to [bend right=15] (Mt);
  
\end{tikzpicture}
\end{center}
  By definition $\widetilde{M} \twoheadleftarrow G_c' \hookrightarrow G_c$ is the pullback of $\widetilde{M} \hookrightarrow M \twoheadleftarrow G_c$.
  The two morphisms $\widetilde{M} \twoheadleftarrow \widetilde{G}_c \hookrightarrow G_c$ form a commutative square with $\widetilde{M} \hookrightarrow M \twoheadleftarrow G_c$.
  The universal property of the pullback yields a mono $\widetilde{G}_c \hookrightarrow G_c'$ and we get an embedding of $\widetilde{e}$ in $e'$ yielding the same $c$-cocycle in $\Ext^c_\A(\widetilde{M},N)$.
\end{proof}

For the induction proof of Proposition~\ref{prop:Cfree} we need the next lemma, which shows how to replace short exact sequences in $\A$ by short exact sequences of $\C$-torsion-free objects.

\begin{lemm} \label{lemm:Cfree}
  Let $\C$ be an almost split thick subcategory of the \nameft{Abel}ian category $\A$ and
  $e:\  0 \xleftarrow{} M \xleftarrow{} G \xleftarrow{} L \xleftarrow{} 0$ a short exact $\A$-sequence with $\C$-torsion-free $M$.
  Then there exists a short exact $\A$-subsequence $e':\  0 \xleftarrow{} M' \xleftarrow{} G' \xleftarrow{} L' \xleftarrow{} 0$ with $\C$-torsion-free objects which equals $e$ up to $\C$-factors. \label{lemm:Cfree.a}
\end{lemm}
\begin{wrapfigure}[7]{r}{3.2cm}
\vspace{-1em}
  \centering
 \begin{tikzpicture}[C/.style={color=red,line width=1.5pt,dotted},new/.style={color=blue,line width=2pt},map/.style={color=gray,dashed}]
    
  \newcommand{\circlewidth}{1.5pt}
  
  \coordinate (a) at (0,1);
  \coordinate (c) at (-0.5,0.3);
  \coordinate (m) at (-4,0);
  
  \coordinate (0_G);
  \coordinate (imL1) at ($(0_G)+(a)$);
  \coordinate (G1) at ($(imL1)+(a)$);
  \coordinate (imL) at ($(imL1)+(c)$);
  \coordinate (imL_plus_G1) at ($(G1)+(c)$);
  \coordinate (G) at ($(imL_plus_G1)+(c)$);
  
  \node at (G) [above right] {$G$};
  \node at (G1) [above right] {$G'$};
  
  \node at (imL) [below left] {$L$};
  \node at ($0.8*(imL1)+0.2*(0_G)$) [right] {$L'$};
  
  \draw[decorate,decoration=brace] ($(G1)+(0.2,0)$) -- node[right] {$M'$} ($(imL1)+(0.2,0)$);
  \draw[decorate,decoration=brace] ($(imL)-(0.9,0)$) -- node[left] {$M$} ($(G)-(0.4,0)$);
  \draw[decorate,decoration=brace] ($(imL)-(0.1,0)$) -- node[left] {$M''$} ($(imL_plus_G1)-(0.1,0)$);
  
  \draw[new] (0_G) -- (G1);
  \draw[new] (imL) -- (imL_plus_G1);
  \draw[C] (G1) -- (G);
  \draw[C] (imL1) -- (imL);
  
  \fill (0_G) circle (\circlewidth);
  \fill (imL1) circle (\circlewidth);
  \fill (G1) circle (\circlewidth);
  \fill (imL) circle (\circlewidth);
  \fill (imL_plus_G1) circle (\circlewidth);
  \fill (G) circle (\circlewidth);
   
 \end{tikzpicture}
\end{wrapfigure}
\mbox{}
\vspace{-2em}
\begin{proof}
We interpret $L$ as a subobject of $G$ with factor object $M$.
Let $G'$ be an almost $\C$-complement in $G$.
Hence $G/G'$ lies in $\C$.
Now define $L'=L \cap G'$ and $M' := G'/L'$.
The object $L'$ is $\C$-torsion-free as a subobject of $G'$ and $M'$ is $\C$-torsion-free since it is isomorphic to the subobject $M''=(G'+L)/L$ of $M=G/L$.
Finally $M/M''$ lies in $\C$ as a factor of $G/G'$ and $L/L'$ lies in $\C$ since it is isomorphic to the subobject $(G'+L)/G'$ of $G/G'$.
\end{proof}

\begin{proof}[Proof of Proposition~\ref{prop:Cfree}]
  We will construct $\widetilde{e}$ by induction on $c$.
  For $c=1$ take $\widetilde{e} = e$ by Lemma~\ref{lemm:extension_of_C_free_is_C_free}.
  Now assume the statement is true for $c-1$ (i.e., for complexes of length $c+1$).
  Define $L := \img(G_c \xleftarrow{} G_{c-1})$.
  Write $e$ as the \nameft{Yoneda} product (i.e., concatenation) $e_1 e_2$ of the short exact $\A$-sequence  $e_1:\  0 \xleftarrow{} M \xleftarrow{} G_c \xleftarrow{} L \xleftarrow{} 0$ and the exact $\A$-complex $e_2:\  0 \xleftarrow{} L \xleftarrow{} G_{c-1} \xleftarrow{} \cdots \xleftarrow{} G_1 \xleftarrow{} N \xleftarrow{} 0$.
  First apply Lemma~\ref{lemm:Cfree} to $e_1$ and obtain the short exact $\C$-torsion-free $\A$-sequence $e_1': 0 \xleftarrow{} M' \xleftarrow{} G_c' \xleftarrow{} L' \xleftarrow{} 0$.
  Now define the exact $\A$-subcomplex
  \[
    e_2':=(L'\hookrightarrow L)e_2:\  0 \xleftarrow{} L' \xleftarrow{} G'_{c-1} \xleftarrow{} G_{c-2} \xleftarrow{} \cdots \xleftarrow{} G_1 \xleftarrow{} N \xleftarrow{} 0
  \]
  of $e_2$.
  By the induction hypothesis there exists an exact $\C$-torsion-free $\A$-subcomplex $\widetilde{e}_2:\  0 \xleftarrow{} \widetilde{L} \xleftarrow{} \widetilde{G}_{c-1} \xleftarrow{} \widetilde{G}_{c-2} \xleftarrow{} \cdots \xleftarrow{} \widetilde{G}_1 \xleftarrow{} N \xleftarrow{} 0$ of $e_2'$.
  Replacing $L'$ by its subobject $\widetilde{L}$ in $e_1'$ we obtain the $\C$-torsion-free $\A$-complex $\widetilde{e}_1': \  0 \xleftarrow{} M' \xleftarrow{} G_c' \xleftarrow{} \widetilde{L} \xleftarrow{} 0$ which is exact up to a $\C$-defect.
  Now apply Lemma~\ref{lemm:Cfree_exact} to $\widetilde{e}_1'$ and obtain the short exact $\C$-torsion-free $\A$-sequence $\widetilde{e}_1:\  0 \xleftarrow{} \widetilde{M} \xleftarrow{} \widetilde{G}_c \xleftarrow{} \widetilde{L} \xleftarrow{} 0$.
  Finally define $\widetilde{e} := \widetilde{e}_1 \widetilde{e}_2$.
  
  All constructions in this proof yield subcomplexes equal to their super-complexes up to $\C$-factors.
  Thus, we conclude the $\widetilde{e}$ equals $e$ up to $\C$-factors.
\end{proof}

By Remark~\ref{rmrk:saturated_form} the following theorem is the equivalent ``saturated form'' of Theorem~\ref{thm:main}.
\begin{theorem} \label{thm:iso}
  If $\C$ is an almost split localizing subcategory of the \nameft{Abel}ian category $\A$ then
  \[
    \widehat{\QQ}^{\Ext}: \varinjlim_{\substack{ M' \leq M,\\ M/M' \in \C}} \Ext^c_\A(M',N) \to \Ext^c_{\Sat_\C(\A)}(M,N)
  \]
  is an isomorphism (of \nameft{Abel}ian groups) for all $\C$-saturated $M,N \in \A$.
\end{theorem}
\begin{proof}
  Let $e' \in \Ext^c_\A(M',N)$ and $e'' \in \Ext^c_\A(M'',N)$ be two cocycles which map to the same element $\widehat{e}:=\widehat{Q}(e') = \widehat{Q}(e'') \in \Ext^c_{\Sat_\C(\A)}(M,N)$.
  By Proposition~\ref{prop:Cfree} we can pass in the colimit to $\C$-torsion-free representatives $\widetilde{e}' \in \Ext^c_\A(\widetilde{M'},N)$ and $\widetilde{e}'' \in \Ext^c_\A(\widetilde{M''},N)$, which are exact $\A$-subcomplexes of $e'$ and $e''$, respectively.
  Furthermore, $\widetilde{e}'$ is an $\A$-subcomplex of $\iota(\widehat{e})$, as it is $\C$-torsion-free and the kernel of $e'\xrightarrow{} \iota(\widehat{e})$ is $H_\C(e')$; the same holds for $\widetilde{e}''$.
  Taking the intersection of $\widetilde{e}'$ and $\widetilde{e}''$ as subcomplexes of $\iota(\widehat{e})$ we obtain an $\A$-subcomplex $\breve{e}$ of $\iota(\widehat{e})$, which is not necessarily exact.
  Lemma~\ref{lemm:Cfree_exact} yields an exact $\A$-subcomplex $e^\perp \leq \breve{e}$ which still represents the same cocycle as $\widetilde{e}'$ and $\widetilde{e}''$ and hence $e'$ and $e''$ in the colimit, and thus all these cocycles are equal in the colimit.
\end{proof}

\begin{appendix}

\section{Sketch of the proper constructive setup} \label{sec:computability}

We now roughly describe the constructive context of this paper.
A detailed description would require a more elaborate preparation and would distract from the main result of this paper, which in this form should already be self-contained.
The standard way to express mathematical notions constructively is to provide algorithms for all disjunctions and all existential quantifiers appearing in the defining axioms of a mathematical structure.
In the case of \nameft{Abel}ian categories this led us to the notion of a \textbf{computable \nameft{Abel}ian} or \textbf{constructively \nameft{Abel}ian category} \cite{BL}.
Given that, all constructions which only depend on a category being \nameft{Abel}ian become computable\footnote{A constructive treatment of spectral sequences along these lines can be found in \cite{BaSF} with a computer implementation in \cite{homalg-package}.}.
The computability of $\A$ implies, in particular, that we can compute in its $\Hom$-groups only \emph{locally}, i.e., we can decide element membership in the $\Hom$-sets, whether morphisms are zero, add and subtract morphisms, and hence decide the equality of two morphisms.
This does not imply that we can ``oversee'' a $\Hom$-group in any way, not even being able to decide its triviality (see $\Hom$-computability below).

For an \nameft{Abel}ian category $\A$ with thick subcategory $\C\subset\A$ we prove in \cite{BL_GabrielMorphisms} that $\A/\C$ is computable once the \nameft{Abel}ian category $\A$ is computable and the membership in $\C \subset \A$ is constructively decidable.

We call $\C \subset \A$ \textbf{constructively localizing} if there exists algorithms to compute the \nameft{Gabriel} monad\footnote{as we call it in \cite{BL_Monads}.} $\SS \circ \QQ$ together with its unit.
Formula \eqref{eq:Hom} in Section~\ref{sec:Hom} proves that if $\A$ is $\Hom$-computable and $\C \subset \A$ is constructively localizing then $\A/\C$ is $\Hom$-computable, where \textbf{$\Hom$-computability} means the computability as an enriched\footnote{Enriched categories are usually required to be small.
In an algorithmic setting any category is small, as the possible states of the computer memory is a set.} category over a \emph{computable} monoidal category\footnote{... of \nameft{Abel}ian groups, $k$-vector spaces, etc., depending on the context.}.

Theorem~\ref{thm:main} implies that $\A/\C$ is $\Ext$-computable if $\A$ is $\Ext$-computable and $\C \subset \A$ is almost split localizing and constructively localizing and the direct limit is constructive.
We would define \textbf{$\Ext$-computability} to be the $\Hom$-computability of the derived category of $\A$.
This would lead too far away.

Finally, we note that the entire proof of Theorem~\ref{thm:main} is constructive and suited for computer implementation.
So if we assume that $\A$ is computable and $\C \subset \A$ is constructively almost split localizing\footnote{I.e., constructively localizing and that we can algorithmically construct the maximal almost $\C$-complement of objects in $\A$} then the proof of Theorem~\ref{thm:main} provides an algorithm to compute images and preimages of elements represented as \nameft{Yoneda} cocycles under $\widehat{\QQ}^{\Ext}: \varinjlim_{\substack{ M' \leq M,\\ M/M' \in \C}} \Ext^c_\A(M',N) \to \Ext^c_{\Sat_\C(\A)}(M,N)$.
Furthermore, if $\A$ is $\Hom$-computable and has constructively enough projectives or injectives then we can decide equality of (\nameft{Yoneda}) cocycles (cf.~\cite[Appendix~B]{BB}).

\end{appendix}

\section*{Acknowledgments}
We thank \nameft{Markus Perling} and \nameft{Greg Smith} for discussions on the range of applicability of Theorem~\ref{thm:main} to coherent sheaves.
We are indebted to anonymous referee who spotted a serious issue in our first version of the injectivity proof.
Addressing it helped us to correct, streamline, and simplify the whole argument.

%% file: Ext_computability.bbl
\def\cprime{$'$} \def\cprime{$'$} \def\cprime{$'$} \def\cprime{$'$}
  \def\cprime{$'$}
\providecommand{\bysame}{\leavevmode\hbox to3em{\hrulefill}\thinspace}
\providecommand{\MR}{\relax\ifhmode\unskip\space\fi MR }
\providecommand{\MRhref}[2]{%
  \href{http://www.ams.org/mathscinet-getitem?mr=#1}{#2}
}
\providecommand{\href}[2]{#2}